\documentclass[a4paper, 11pt, reqno]{article} %

\usepackage{amsmath, amsthm, amssymb, amsfonts,  physics, braket}
\usepackage{geometry,  outlines, array, verbatim, authblk, enumitem, bbold}
\usepackage{float, tabularx, longtable, graphicx, wrapfig, subcaption, makecell, fancyhdr}
\usepackage{xcolor}

\usepackage[alphabetic,backrefs]{amsrefs}
\usepackage[linktocpage=true,colorlinks,citecolor=blue,linkcolor=blue,pagebackref]{hyperref}
\usepackage{cleveref}

\newtheorem{thm}{Theorem}[section]
\newtheorem{prop}[thm]{Proposition}
\newtheorem{lem}[thm]{Lemma}

\theoremstyle{definition}
\newtheorem{defn}[thm]{Definition}
\newtheorem{eg}[thm]{Example}
\newtheorem{conv}{Convention}
\newtheorem{rem}[thm]{Remark}

\theoremstyle{remark}

\numberwithin{equation}{subsection}

\newcommand{\s}[1]{ \displaystyle \sum_{i=1}^{#1}}
\newcommand{\h}{\mathcal{H} }

\newcommand{\C}{\mathbb{C}}

\newcommand{\A}{\mathcal{A}}

\newcommand{\G}{\mathcal{G}}

\newcommand{\M}{\mathcal{M}}
\newcommand{\N}{\mathcal{N}}
\newcommand{\inner}[2]{\left \langle {#1} , {#2} \right \rangle }

\newcommand{\unit}{\mathbb{1}}

\title{ \vspace{-1.5cm} \textbf{Spectral bounds for the quantum chromatic number of quantum graphs}}
\author{\vspace{0.5cm} PRIYANGA GANESAN\thanks{Department of Mathematics, Texas A\&M University, College Station. \textit{Email: priyanga.g@tamu.edu}} }

\date{ \vspace {-0.2cm}  November 30, 2021}

\begin{document}

\maketitle
\vspace{-0.7cm}

\begin{abstract}
Quantum graphs are an operator space generalization of classical graphs that have emerged in different branches of mathematics including operator theory, non-commutative topology and quantum information theory.
In this paper, we obtain lower bounds for the classical and quantum chromatic number of a quantum graph using eigenvalues of the quantum adjacency matrix. In particular, we prove a quantum generalization of Hoffman's bound and introduce quantum analogues for the edge number, Laplacian and signless Laplacian. We generalize all the spectral bounds of Elphick \& Wocjan \cite{EW1} to the quantum graph setting and demonstrate the tightness of these bounds in the case of complete quantum graphs.  Our results are achieved using techniques from linear algebra and a combinatorial definition of quantum graph coloring, which is obtained from the winning strategies of a quantum-to-classical nonlocal graph coloring game \cite{bgh}. 

\end{abstract}

\section{Introduction}

Graph coloring has been well-studied in mathematics since the eighteenth century and has widespread applications in day-to-day life, including scheduling problems, register allocation, radio frequency assignments and sudoku solutions \cite{coloring-applications}. Traditionally, the coloring of a graph refers to an assignment of labels (called colors) to the vertices of a graph such that no two adjacent vertices share the same color. The chromatic number of a graph is defined to be the minimum number of colors for which such an assignment is possible. 

More recently, a quantum generalization of the chromatic number was introduced within the framework of non-local games in quantum information theory \cite{cameron}.  The quantum chromatic number of a graph is defined as the minimal number of colors necessary in a quantum protocol in which two separated players, who cannot communicate with each other but share an entangled quantum state, try to convince an interrogator with certainty that they have a coloring for the given graph. There are known examples of graphs  whose quantum chromatic number is strictly smaller than its classical chromatic number \cite{mr1, cameron}, thus exhibiting the power of quantum entanglement. Quantum chromatic number of classical graphs have close connections to Tsirelson's conjecture and the Connes embedding problem and have been extensively studied in the past decade \cite{rank1_qcn, mr1, paulsen-ivan, paulsen-ivan-sev-winter-stahlke}. In general, computing the chromatic number of a graph is an NP-hard problem.
However, several lower bounds on the (quantum) chromatic number have been shown using spectral graph theory \cite{EW1} .
In this paper, we are interested in generalizing these spectral bounds to the setting of \textit{quantum graphs} and estimating the quantum chromatic number of a quantum graph. 

Quantum graphs are a non-commutative generalization of classical graphs that have attracted significant attention in recent years due to their intriguing connections with several areas of mathematics, physics and computer science. They first appeared in \cite{shulman}, and have independently emerged in other disguises thereafter. In information theory, quantum graphs were introduced as a quantum analogue of the confusability graph of classical channels \cite{dsw}. Another definition was proposed in the context of quantum relations \cite{weaver-qg}, which describes a quantum graph as a reflexive and symmetric quantum relation on a finite dimensional von-Neumann algebra. In \cite{oxford}, an equivalent perspective on quantum graphs was developed in a categorical framework for quantum functions, using a quantum adjacency matrix, and was further generalized in \cite{kari}.  In recent years, research in quantum graph theory has undergone vast developments and quantum graphs have been explored in the context of zero-error quantum information theory, quantum error correction, operator algebras, non-local games, quantum symmetries, non-commutative topology and other fields \cite{kari2, randomqg, ramsey1, ramsey2, matsuda, gromada}.
In particular, there have been multiple studies  on the coloring of quantum graphs \cite{paulsen-ortiz, mehta, stahlke, bgh, ivan-lyudmila} leading to 
different variants of the chromatic number of a quantum graph, in both the classical and quantum sense.

One recent approach developed in \cite{bgh} defines the coloring of a quantum graph using a two-player non-local game involving quantum inputs and classical outputs. This game generalizes the coloring game for classical graphs \cite{cameron} and introduces chromatic numbers of a quantum graph in different mathematical models: $loc, q, qa, qc, C^*, hered, alg$.
It was shown that the chromatic numbers defined in this framework agree nicely with other versions in the literature \cite{mr1, stahlke, mehta}, and also lead to a four-coloring theorem for quantum graphs in the algebraic model. We adopt this formalism of quantum graph coloring in the present paper. 

The goal of this paper is to obtain bounds for the quantum chromatic number of quantum graphs. 
Chromatic numbers of quantum graphs are closely related to the zero-error capacity of quantum channels \cite{dsw}. Hence, estimating these numbers is useful for the development of zero-error quantum communication.  In \cite{EW1}, the authors proved many lower bounds on the quantum chromatic number of \textit{classical graphs} using an algebraic characterization of graph coloring. 
We extend their results to the setting of quantum graphs using a combinatorial definition of quantum graph coloring developed in \cite{bgh}. 
Our approach uses the quantum adjacency matrix, defined in \cite{oxford, kari}, to associate a spectrum with the given quantum graph. We use this spectrum and techniques adapted from \cite{EW1} to achieve the spectral estimates. In this process, we naturally get lower bounds for the classical chromatic numbers of quantum graphs since the classical chromatic number is greater than or equal to the quantum chromatic number. 

Our main result can be summarized as follows:

\begin{thm} \label{main result}
Let $\G =  (\M, \psi, A, S)$ be an irreflexive quantum graph, and let $\chi(\G)$ and $\chi_q(\G)$ denote the classical and quantum chromatic numbers of $\G$ respectively. Then,
\[
1 + \max \left\{  
\dfrac{\lambda_{\max}}{|\lambda_{\min}|}, 
\dfrac{\dim(S)}{\dim(S) - \dim(\M) \gamma_{\min}}, 
\dfrac{s^\pm}{s^\mp}, 
\dfrac{n^\pm}{n^\mp} ,
\dfrac{\lambda_{\max}}{\lambda_{\max} - \gamma_{\max}+ \theta_{max}} 
\right \}
\le 
\chi_q(\G) \le \chi(\G).
\]
\end{thm}

Specifically, we prove that Hoffman's bound \cite{hoffman} holds in the case of quantum graphs. 
We also introduce quantum analogues for the edge number, inertia, Laplacian and signless Laplacian of a quantum graph along the way.
Further, we demonstrate the tightness of all the bounds in the case of irreflexive complete quantum graphs. 

\noindent 
Our paper is organized as follows:  Section \S \ref{background} provides the necessary background on quantum graphs and the connections between different perspectives. We also review the notion of quantum graph coloring and chromatic numbers in this section. Section \S \ref{pinching and twirling} introduces the spectrum of a quantum graph and develops algebraic results connecting the quantum adjacency operator to quantum graph coloring. In section \S \ref{bounds}, we prove the spectral lower bounds listed in theorem \ref{main result} for the quantum chromatic numbers of quantum graphs. 
We conclude with an illustration of the bounds in the case of complete quantum graphs in section \S \ref{application}.
The appendix \S \ref{appendix} presents a table translating the different definitions of quantum graphs.

\section{Preliminaries} \label{background}

In this section, we review some definitions and results on quantum graphs and quantum coloring which will be required for our discussion. We begin by listing some notations used in the paper.

\subsection{Notations}
\begin{itemize}[noitemsep]
\item $[n]$ denotes the discrete set $\{1,2, \ldots, n\}$.
\item $\ket{\cdot}$ denotes a column vector, while $\bra{\cdot}$ denotes its conjugate transpose.
\item $M_n$ denotes the set of all $ n \times n$ complex matrices.
\item $e_{ij}$ denotes the matrix whose $i^{th}$ row - $j^{th}$ column has entry $1$ and all other entries are 0.
\item $\Tr$ denotes the natural trace, given by summing all diagonal terms of a
matrix.
\item $B(\h)$ denotes the algebra of bounded linear operators on a Hilbert space $\h$.
\item If $T$ is a subset of an algebra $\A$, then the commutant of $T$ is denoted by $T' = \{ a \in \A: at = ta, \; \forall \; t \in T\}$.
\item The spectrum of an operator $A$ is denoted by $\sigma(A)$.
\item $G$ denotes a classical graph, $\chi(G)$ denotes the classical chromatic number of $G$ and $\chi_q(G)$ denotes the quantum chromatic number of $G$.
\end{itemize}

\subsection{Quantum graphs as operator spaces}

Quantum graphs can be defined in different ways, as mentioned in the introduction. One way to describe them is as operator spaces satisfying a certain bimodule property \cite{weaver-qg}. This approach is commonly used for studying quantum coloring problems. 
We discuss this formalism here:

\begin{defn} \label{qg1}
Let $\h$ be a complex Hilbert space and $\M \subseteq B(\h)$ be a (non-degenerate) von Neumann algebra.
A \textit{quantum graph} on $\M$ is an operator space $S \subseteq B(\h)$ that is closed under adjoint and is a bimodule over the commutant of $\M$, that is $\M' S \M' \subseteq S$. We denote this quantum graph by the tuple $\G = (S, \M, B(\h))$.
\end{defn}

Motivated by confusability graphs in information theory, quantum graphs are generally assumed to be reflexive ($I \in S$) and hence, $S$ is an operator system in $B(\h)$. But for the purposes of graph coloring, we will only consider irreflexive quantum graphs, that is quantum analogues of graphs without loops.

\begin{defn}
A quantum graph $(S, \M, B(\h))$ is said to be \textit{irreflexive} if $S \subseteq (\M')^{\perp}$.
\end{defn}
 
In particular, an irreflexive quantum graph on $M_n$ (with the standard representation $M_n = B(\C^n)$) is simply a self-adjoint traceless operator subspace in $M_n$. This is sometimes used as the definition of non-commutative graphs in the literature \cite{stahlke}.

\noindent It can be shown that the operator space $S$ associated to a quantum graph $(S, \M, B(\h))$
is essentially independent of the representation of $\M$ \cite{weaver1}.
The intuition is that $S$ contains operators that represent edges in the graph, as illustrated by the following example.

\begin{eg} \label{classical S}
Let $G$ be a classical graph on $n$ vertices. One can identify the vertex set of $G$ with the algebra of diagonal matrices $D_n \subseteq M_n$, by identifying each vertex $i$ with the diagonal matrix $e_{ii} \in D_n$. Then, $S_G = span\{ e_{ij}: (i,j) \mbox{ is an edge in } G\} \subseteq M_n$ is a quantum graph over $D_n$. 
\end{eg}

\begin{rem}
Any quantum graph over $D_n$ is necessarily of the form $S_G$ for some classical graph $G$ \cite{weaver-qg}. Also, two reflexive classical graphs $G_1, G_2$ are isomorphic if and only if their corresponding operator systems $S_{G_1}, S_{G_2}$ are unitally completely order isomorphic \cite{paulsen-ortiz}.
\end{rem}

A ``purely quantum" example is the following one:

\begin{eg}
Let $\M = M_2$ and $S = \left \{ 
\begin{bmatrix}
a & b \\
c & a 
\end{bmatrix} :
a, b, c \in \C
\right \}$. Then $(S, M_2, B(\C^2))$ is a quantum graph on $M_2$ that doesn't arise from any classical graph.
\end{eg}

\subsection{Quantum graphs with a quantum adjacency matrix}
 
In this paper, we take advantage of an alternate (but equivalent) definition of a quantum graph, which involves quantizing the vertex set and the adjacency matrix. This formalism was first introduced in \cite{oxford} using the language of special symmetric dagger Frobenius algebras, and was later generalized to the non-tracial case in \cite{kari, matsuda}.  In this perspective, the non-commutative analogue of a vertex set is played by a C*-algebra, which also carries the structure of a Hilbert space. It is defined as follows:

\begin{defn}[Quantum set]
A quantum set is a pair $(\M, \psi)$, where $\M$ is a finite dimensional C*-algebra and $\psi: \M \to \C$ is a faithful state.
\end{defn}

Using $\psi$, one can view $\M$ as a Hilbert space $L^2(\M) = L^2(\M, \psi)$ obtained from the GNS representation of $\M$ with respect to $\psi$. That is, $L^2(\M)$ is the vector space $\M$ equipped with the inner product $\langle x, y \rangle = \psi(y^*x)$. 

Let $m: \M \otimes \M \to \M$ denote the multiplication map and $m^*$ denote the adjoint of $m$ when viewed as a linear operator from $L^2(\M) \otimes L^2(\M) \to L^2(\M)$. Further, we denote the unit of $\M$ by $\unit$ and let $\eta: \C \to \M$ be the unit map $\lambda \mapsto \lambda \unit$. The adjoint of $\eta$ (as an operator on Hilbert spaces) is denoted by $\eta^*$ and is equal to $\psi$.
While there are many choices for a faithful state $\psi$ on $\M$, we will restrict our attention to $\delta$-forms, as done in \cite{kari}. 

\begin{defn}
For $\delta > 0$, a state $\psi: \M \to \C$ is called a \textit{$\delta$-form} if $mm^* = \delta^2 I$. 
\end{defn}

\begin{eg} 
Let $X$ be a finite set and $\M = C(X)$ be the algebra of continuous complex valued functions on $X$ . Then the uniform measure $\psi (f) = \frac{1}{|X|}   \sum_{x \in X} f(x)$ is a $\delta$-form on $C(X)$ with $\delta^2 = |X|$. In this case, $m^*$ is given by $e_i \mapsto |X| (e_i \otimes e_i)$, where $e_i$ is the characteristic function on the set $\{i \} \subseteq X$.
\end{eg}

\begin{eg}
Let $\M$ be $M_n$ equipped with the canonical normalized trace $\psi = \frac{1}{n} \Tr$. 
Then $m^*(e_{ij}) = n \sum_{k=1}^n e_{ik} \otimes e_{kj}$, and $\psi$ is an $n$-form on $M_n$.
\end{eg}

The $\delta$-forms in the above examples are tracial, that is $\psi(xy) = \psi(yx)$ for all $x,y \in \M$. 
A tracial $\delta$-form on a finite dimensional C*-algebra is unique and has a nice form, which will be used in later sections.
We recall this now:

\begin{prop} \label{plancheral}
Let $\M$ be a finite dimensional C*-algebra, decomposed as $\M \cong \bigoplus_{i=1}^N M_{n_i}$, where $N, n_1, n_2, \ldots, n_N$ are some positive integers.
Then, there exists a \textit{unique} tracial $\delta$-form on $\M$ given by 
\begin{equation} \label{plancheral trace}
\psi = \dfrac{1}{dim(\M)} \bigoplus_{i=1}^N n_i \Tr( \cdot ) 
\end{equation}
In this case, $\delta^2 = dim(\M)$ and the state $\psi$ is called the \textit{Plancheral trace}.
Moreover, $\psi = \frac{1}{\dim(\M)} \Tr\vert_{\M}$ , where $\Tr: B(L^2(\M)) \to \C$ is the canonical trace.
\end{prop}

A quantum set endowed with an additional structure of a quantum adjacency matrix yields a quantum graph. 

\begin{defn}[\cite{kari}] \label{qg2}
Let $\M$ be a finite dimensional C*-algebra equipped with a $\delta$-form $\psi$. A self-adjoint linear map $A: L^2(\M) \to L^2(\M)$ is called a \textit{quantum adjacency matrix} if it satisfies the following conditions:
\begin{enumerate}
\item $m(A \otimes A)m^* = \delta^2 A$,
\item $(I \otimes \eta^*m )(I \otimes A \otimes I) ( m^* \eta \otimes I) = A $.
\end{enumerate}
The tuple $\G = (\M, \psi, A)$ is called an (undirected) quantum graph. 

The quantum graph $(\M, \psi, A)$ is said to be \textit{reflexive} if it further satisfies the condition $m(A \otimes I)m^* = \delta^{2} I$ or is said to be \textit{irreflexive} if it satisfies the condition $m(A \otimes I)m^* = 0$.

\end{defn}

The motivation for the above definition comes from the commutative setting where $\M = C(X)$ and $\psi$ is the uniform measure on $X$. In this case, the quantum adjacency matrix $A: L^2(\M) \to L^2(\M)$ can be identified with a matrix in $M_{|X|}(\C)$, and the operation $\delta^{-2} m(P \otimes Q)m^*$ is simply the schur product of the matrices $P$ and $Q$, given by entrywise multiplication. So, the first condition in definition \ref{qg2} says that $A$ must be an idempotent with respect to Schur multiplication, which is equivalent to saying that $A$ has entries in $\{0, 1\}$. The second condition says $ A = A^T$. 
If we drop the second condition in definition \ref{qg2}, it is called a \textit{directed} quantum graph \cite{kari2}. 

\begin{rem} \label{* preserve}
The self-adjointness of $A$ along with condition (2) in definition \ref{qg2} implies that $A$ is also *-preserving  \cite{matsuda}, that is 
$Ax^* = (Ax)^*$ for all $x \in \M$.
\end{rem}

Every quantum set can be easily equipped with an adjacency operator to obtain a quantum graph. An example is the 
complete quantum graph.

\begin{defn}
Let $(\M, \psi)$ be a quantum set.  A reflexive \textit{complete quantum graph} on $\M$ is defined by $A =  \delta^2 \psi (\cdot) \unit$. In the classical case, this gives the all 1s matrix and corresponds to the reflexive complete graph on $\dim(\M)$ vertices.

An \textit{irreflexive} complete quantum graph on $(\M, \psi)$ is defined by $A =  \delta^2 \psi (\cdot) \unit - I$.

\end{defn}

There are several non-trivial examples of quantum graphs. In particular, \cite{matsuda} gives a concrete classification of all undirected reflexive quantum graphs on $M_2$, and \cite{gromada} gives an example of a quantum graph, which is not quantum isomorphic to any classical graph.

\subsection{Translation between different perspectives of quantum graphs} \label{translation}
The two definitions of quantum graphs given in \ref{qg1} and \ref{qg2} can be shown to be equivalent \cite{oxford}, using a bijective correspondence between linear operators on $L^2(\M)$ and elements in $\M \otimes \M^{op}$.
A thorough algebraic proof for the correspondence between different definitions of quantum graphs is given in \cite{larissa}.
We summarize this connection below:

\begin{enumerate}
\item Given a quantum graph $(\M, \psi, A)$, define $P: B(L^2(\M)) \to B(L^2(\M))$ as 
\begin{equation} \label{a to p}
P(X) = \delta^{-2} m(A \otimes X)m^* .
\end{equation} 
Then, range$(P)$ is a self-adjoint operator subspace in $B(L^2(\M))$ that is a bimodule over $\M'$.

\item Given a quantum graph $( S, \; (\M, \psi), \; B(L^2(\M)) \;)$, 
let $P: B(L^2(\M)) \to B(L^2(\M))$ denote a self-adjoint $\M'-\M'$ bimodule projection with $range(P) = S$. 

That is, $P(axb) = aP(x)b$, for all $x \in B(L^2(\M))$, $a,b \in \M'$ and $P^2 = P = P^*$, where the adjoint is taken with respect to the trace inner product on $B(L^2(\M))$.
(Such a $P$ always exists and is unique for the given $S$ \cite{weaver1}.)

Then, $A: L^2(\M) \to L^2(\M)$ defined by
\begin{equation} \label{p to a}
A(x) = \delta^2 (\psi \otimes I) P ( x \otimes 1)
\end{equation}
is a quantum adjacency matrix on $(\M,\psi)$.
\end{enumerate}

\begin{rem} \label{cbm}
In \eqref{p to a} $P$ is interpreted as an element of $\M \otimes \M^{op}$ using the following well-known *-isomorphism in finite dimensions:
\begin{eqnarray*} \label{p and P}
\M \otimes \M^{op} & \cong & {}_{\M'}CB_{\M'} (B(L^2(\M))), \mbox{ given by } \\
x \otimes y^{op} & \longleftrightarrow & x(\cdot)y.
\end{eqnarray*}
Here, $\M^{op}$ denotes the opposite algebra of $\M$ and ${}_{\M'}CB_{\M'} (B(L^2(\M)))$
denotes the set of completely bounded maps $P$ on $B(L^2(\M))$ with the property 
$P(axb) = aP(x)b$, for all $x \in B(L^2(\M))$, $a,b \in \M'$. 
An infinite dimensional version of this result can be found in \cite{effros-ruan}.
\end{rem}

\begin{rem}
We also note that the expressions \eqref{a to p} and \eqref{p to a} are inverses of each other.
\end{rem}

In general, the above correspondence between $S$ and linear operators $A$ is not one-one since there are several different bimodule idempotents $P$ with the same range $S$. 
However, there is a \textit{unique} self-adjoint quantum adjacency matrix $A$ for a given $S$, which corresponds to the unique orthogonal bimodule projection onto $S$. In this case, $A$ is also completely positive, which was used as an alternate definition of quantum adjacency matrix in \cite{randomqg}.

\subsection{Chromatic number of quantum graphs}

In this section, we review a notion of quantum graph coloring that was developed in \cite{bgh} using a two-player quantum-to-classical nonlocal game, generalizing the coloring game for classical graphs \cite{cameron}.
The inputs for the quantum graph coloring game are elements from a suitably chosen basis for the graph operator space (known as quantum edge basis) and the outputs are classical color assignments. The inputs are quantum in the sense that they are tensor product states, where one player receives the left leg of the tensor and the other player receives the right leg. The players win the game if their responses jointly satisfy a synchronicity condition and respect the adjacency structure of the quantum graph. We refer the reader to \cite{bgh} for more details on the game, and for the results presented in this section.

Using the winning strategies for the quantum graph coloring game, the chromatic number of a quantum graph can be defined in different mathematical models: $loc, q, qa, qc, C^*, hered, alg$.  In this paper, we will restrict our discussion to the classical $(loc)$, and quantum $(q)$ chromatic numbers. 
We begin with recalling an algebraic definition of quantum graph coloring that arises from the non-local game in \cite{bgh}.

\begin{defn}[ \cite{bgh}] \label{qcoloring}
Let $\G = (S, \M, B(\h))$ be an irreflexive quantum graph. We say that there is a $c$-coloring of $\G$ if there exists a finite von-Neumann algebra $\N$ with a faithful normal trace and projections $\{P_a\}_{a=1}^c \subseteq \M \otimes \N$ such that 
\begin{enumerate}
\item $P_a^2 = P_a = P_a^*$, for $1 \le a \le c$,
\item $\sum_{i=1}^c P_a = I_{\M \otimes \N}$,
\end{enumerate}
satisfying the following condition:
\begin{equation} \label{annihilate S}
P_a (X \otimes I_{\N}) P_a = 0, \; \forall X \in S \mbox{ and } 1 \le a \le c.
\end{equation}

If $\dim(\N) = 1$, we call it a \textit{classical} $(loc)$ coloring  of $\G$ and if $\dim(\N) < \infty$, we call it a \textit{quantum} $(q)$ coloring  of $\G$.
More generally, when $\N$ is a finite von-Neumann algebra (possibly infinite dimensional), 
it is called a \textit{quantum commuting} $(qc)$ coloring  of $\G$.
\end{defn}

The projections $\{P_a\}_{a=1}^c$ are obtained from the winning strategies of the non-local quantum graph coloring game. In particular, when $\M = D_n$, we recover the usual classical and quantum coloring of classical graphs on $n$ vertices.

\begin{defn}[ \cite{bgh}] \label{qc-qg}
Let $\G = (S, \M, B(\h))$ be an irreflexive quantum graph. The \textit{quantum chromatic number} of $\G$ is defined to be the least $c$ such that there exists a $c$-coloring of $\G$ in the sense of definition \ref{qcoloring}, with $\dim(\N) < \infty$. We denote this quantum chromatic number by $\chi_q(\G)$.
Moreover, when $\dim(\N) = 1$, it is called the \textit{classical} chromatic number $\chi(\G) = \chi_{loc}(\G)$ and when $dim(\N) = \infty$, it is called the \textit{quantum commuting} chromatic number $\chi_{qc}(\G)$.
\end{defn}

It was shown in \cite{bgh} that every quantum graph $\G = (S, \M, M_n)$ has a finite quantum coloring and $\chi_q(\G) \le \dim(\M)$. Further, for all quantum graphs $\G$, we have
\begin{equation}
\chi_{qc} (\G) \le \chi_{q} (\G) \le \chi (\G).
\end{equation}
Also, if $(S, \M, M_n)$ and $(T, \M, M_n)$ are two quantum graphs such that $S \subseteq T$, then  $\chi_t(S, \M, M_n) \le \chi_t(T, \M, M_n)$, where $t \in \{loc, q, qc\}$.

\begin{eg} 
Let $G$ be a classical graph on $n$ vertices and $\G = (S_G, D_n, M_n)$ be the quantum graph associated with $G$, as in example \ref{classical S}. Then, 
\begin{equation}
\chi(\G) = \chi(G) \mbox{ and } \chi_q(\G) = \chi_q(G).
\end{equation}
\end{eg}


\begin{eg}
For complete quantum  graphs, the quantum chromatic number is the full dimension of the quantum vertex set. 
That is, $\chi_q(M_n, \M, M_n) = dim(\M)$. 

\end{eg}

\begin{rem}
Indeed, $\chi(M_n, \M, M_n) < \infty$ if and only if $\M$ is abelian.
In particular, if $\M$ is non-abelian, then $\chi(M_n, \M, M_n) \ne \chi_q(M_n, \M, M_n)$.
\end{rem} 

It is also useful to note that definition \ref{qc-qg} is a special case of Stahlke's entanglement-assisted chromatic number \cite{stahlke}. Also, when $\N = \C$, it is equivalent to Kim \& Mehta's strong chromatic numbers of non-commutative graphs \cite{mehta}.

\section{Use of quantum adjacency matrix in coloring} \label{pinching and twirling}

While definition \ref{qg1} was used in \cite{bgh} for developing chromatic number of quantum graphs, definition \ref{qg2}  offers the advantage of associating a spectrum with the quantum graph, which is useful for estimating these chromatic numbers.  We introduce this now:

\begin{defn}
Let $\M$ be a finite dimensional C*-algebra equipped with its tracial $\delta$-form $\psi$, 
and let $\G = (S, \M, B(L^2(\M,\psi)))$ be a (undirected) quantum graph on $(\M, \psi)$. The \textit{spectrum of $\G$} is defined to be the spectrum of the quantum adjacency operator $A$, defined by
\begin{equation} \label{fix A}
A = \delta^{-2} (\psi \otimes I) P_S (I \otimes \unit),
\end{equation}
where $P_S$ is the orthogonal bimodule projection onto $S$.
\end{defn}

Note that $A$ is \textit{self-adjoint} 
and so, the spectrum of an undirected quantum graph is real.

\begin{conv} \label{assumption}
For the remainder of this paper, $\M$ denotes a finite dimensional C*-algebra equipped with its \textit{tracial} $\delta$-form $\psi$, as given in \ref{plancheral}. 
We assume that our quantum graph $(S, \M, B(L^2(\M,\psi)))$ is irreflexive. 
Further, $A$ always refers to the unique self-adjoint quantum adjacency matrix 
associated with $S$. We denote this quantum graph by $\G = (\M, \psi, A, S)$.
 \end{conv}

We now show the connection between quantum adjacency matrix and quantum graph coloring by generalizing some algebraic results in \cite{EW1} to the quantum graph setting. The following lemma proves that ``pinching" operation annihilates the quantum adjacency matrix and leaves the commutant of the quantum vertex set invariant.

\begin{lem} \label{pinching1}
Let $\G = (\M, \psi, A, S)$ be an irreflexive quantum graph. If $\{P_k\}_{k=1}^c \subseteq \M \otimes \N$ is an arbitrary $c$-quantum coloring of $\G$, then 
\begin{equation} \label{annihilate A}
\sum_{k =1}^c P_k (A \otimes I_{\N}) P_k  = 0, 
\end{equation}
\begin{equation} \label{p commute d}
\sum_{k = 1}^c P_k (E \otimes I_{\N}) P_k  = E \otimes I_{\N}, \;\; \forall E \in \M'.
\end{equation}
\end{lem}

\begin{proof}
We first show that $A \in S$. 
Recall that $A$ is given by \eqref{fix A}, using the orthogonal bimodule projection onto $S$.
Using the inverse relations \eqref{a to p} and \eqref{p to a},
it can be shown that $P_S$ must be of the form $\delta^{-2} m( A \otimes (\cdot) )m^*$.
In particular, $P_S(A) = \delta^{-2} m(A \otimes A)m^* = A$ by the Schur idempotent property of $A$. So, $A \in range(P_S) = S$. Now, by \eqref{annihilate S}, we get that $\sum_{k =1}^c P_k (A \otimes I_{\N}) P_k  = 0$.

\noindent Equation \eqref{p commute d} follows from the fact that the projections $P_k \in \M \otimes \N$ commute with $E \otimes I_{\N} \in \M' \otimes \N'$, and $\sum_{k=1}^c P_k = I_{\M \otimes \N}$.
\end{proof}

The next lemma is a corresponding result for the ``twirling" operation.

\begin{lem} \label{pinching2}
Suppose $\G = (\M, \psi, A, S)$ is an irreflexive quantum graph and $\{P_k\}_{k=1}^c \subseteq \M \otimes \N$ is a $c$-quantum coloring of $\G$. Define $U := \sum_{l =1}^c \omega^l P_l$, where $\omega = e^{2 \pi i / c}$ is a $c^{th}$ root of unity. Then,  
\begin{equation}
\sum_{k =1}^c P_k (X \otimes I_{\N}) P_k  = \dfrac{1}{c} \sum_{k =1}^c U^k (X \otimes I_{\N}) (U^*)^k, 
\;\;\; \forall \; X \in B(L^2(\M)). 
\end{equation}
In particular,
\begin{equation} \label{l7}
\sum_{k =1}^c U^k (A \otimes I_{\N}) (U^*)^k  = 0, 
\end{equation}
\begin{equation}
\sum_{k =1}^c U^k (E \otimes I_{\N}) (U^*)^k  = c \; (E \otimes I_{\N}), \;\; \forall E \in \M'.
\end{equation}

\end{lem}

\begin{proof}
Note that $U^* = \sum_{l=1}^c \omega^{-l} P_l$ since $\{P_l\}_{l =1}^c$ are self-adjoint.
Also, the $k^{th}$ power of $U$ is given by
\[ U^k = \sum_{l=1}^c \omega^{lk} P_l \]
as the projections $\{P_l\}_{l =1}^c$ are mutually orthogonal, that is $P_i P_j = 0$ if $i \ne j$.
Now, for $X \in B(L^2(\M))$, we obtain:

\begin{eqnarray*}
\sum_{k =1}^c U^k (X \otimes I_{\N}) (U^*)^k & = & \sum_{k=1}^c \sum_{l,l' =1}^c \omega^{(l - l')k} P_l (X \otimes I_{\N}) P_{l'} \\
& = &  \sum_{l,l' =1}^c ( \sum_{k=1}^c \omega^{(l - l')k} ) P_l (X \otimes I_{\N}) P_{l'} \\
& = &  \sum_{l,l' =1}^c ( c \; \delta_{l,l'} ) P_l (X \otimes I_{\N}) P_{l'}, \mbox{ where $\delta_{l,l'}$ denotes the Kr\"{o}necker delta}\\
& = & c \sum_{l =1}^c P_l (X \otimes I_{\N}) P_l
\end{eqnarray*}
Hence, we get the result. The rest follows from lemma \ref{pinching1}.
\end{proof}

Next, we note some obvious properties of $A \otimes I_{\N}$ for future reference.

\begin{prop} \label{A tilda}
Suppose $\G = (\M, \psi, A, S)$ is an irreflexive quantum graph and $\{P_k\}_{k=1}^c \subseteq \M \otimes \N$ is an arbitrary $c$-quantum coloring of $\G$. Assume that $2 \le \dim(\M) < \infty$ and $\N \subseteq B(\h)$ for some Hilbert space $\h$, say $\dim(\h) = d$.

Define $\tilde{A} = A \otimes I_{\N}$. Then 
\begin{enumerate}
\item $\tilde{A}$ is self-adjoint and has real eigenvalues.
\item The spectrum of $\tilde{A}$ has the same elements as the spectrum of $A$, but each with a multiplicity of $d$. In particular, the largest and smallest eigenvalue of $\tilde{A}$ coincide with the largest and smallest eigenvalue of $A$, respectively.
\item  $\tilde{A} = \sum_{a,b=1}^c P_a \tilde{A} P_b$. 
\item $\tilde{A}$ can be expressed as a block partitioned matrix $\begin{bmatrix}
\widehat{A}_{11} & \widehat{A}_{12} & \ldots & \widehat{A}_{1c} \\
\widehat{A}_{21} & \widehat{A}_{22} & \ldots & \widehat{A}_{2c} \\
\vdots & \vdots & \vdots & \vdots \\
\widehat{A}_{c1} & \widehat{A}_{c2} & \ldots & \widehat{A}_{cc}
\end{bmatrix}$, such that $\widehat{A}_{ii} = 0$ for all $i \in [c]$.
In particular, $\Tr(A) = \dfrac{1}{d} \Tr(\tilde{A}) = 0$.
\end{enumerate}
\end{prop}

\begin{proof}
The first two statements are evident since $A$ is self-adjoint and tensoring with identity only produces more copies of the same eigenvalues. The third statement follows from the fact that $\sum_{k=1}^c P_k = I_{\M \otimes \N}$.

To see the last statement, note that $\tilde{A}$ can be interpreted as a giant matrix over complex numbers as $\M$ and $\N$ are finite dimensional. Choose an orthonormal basis for $L^2(\M) \otimes \h$ such that all the projections $P_k$ are represented as diagonal matrices. Identify $\widehat{A}_{ab}$ with the matrix $P_a \tilde{A} P_b$. Then, we get the desired block partition. From \eqref{annihilate A}, it follows that $\widehat{A}_{ii} = 0$ for $1 \le i \le c$.
\end{proof}

\section{Spectral lower bounds for $\chi_q(\G)$ and $\chi(\G)$} \label{bounds}

In this section, we obtain spectral lower bounds for the quantum chromatic number of quantum graphs, generalizing results from
\cite{EW1}.
Since $\chi_q(\G) \le \chi(\G)$ \cite{bgh}, our estimates are also lower bounds for the classical chromatic number of quantum graphs.

\noindent Our spectral bounds for an undirected quantum graph $\G = (\M, \psi, A, S)$ can be summarized as follows:
\begin{equation} \label{all bounds}
1 + \max \left\{  
\dfrac{\lambda_{\max}}{|\lambda_{\min}|}, 
\dfrac{\dim(S)}{\dim(S) - \dim(\M) \gamma_{\min}}, 
\dfrac{s^\pm}{s^\mp}, 
\dfrac{n^\pm}{n^\mp} ,
\dfrac{\lambda_{\max}}{\lambda_{\max} - \gamma_{\max}+ \theta_{max}} 
\right \}
\le 
\chi_q(\G) \le \chi(\G).
\end{equation}

\noindent Here, 
$\lambda_{\max}, \lambda_{\min}$ denote the maximum and minimum eigenvalues of $A$;
$s^+, s^-$ denote the sum of the squares of the positive and negative eigenvalues of $A$ respectively;
$n^+, n^-$ are the number of positive and negative eigenvalues of $A$ including multiplicities;
$\gamma_{\max}, \gamma_{\min}$ denote the maximum and minimum eigenvalues of the 
signless Laplacian operator (definition \ref{q-defn}); and $\theta_{\max}$ denotes the maximum eigenvalue of the Laplacian operator (definition \ref{q-defn}).

The key ingredient in proving these bounds is lemma \ref{pinching1} and \ref{pinching2}. Using these, the proof of the corresponding classical bounds can essentially be adapted to our setting. 

Throughout our discussion, we follow convention \ref{assumption}. So, $A$ always refers to the unique self-adjoint quantum adjacency matrix associated with $(S, \M, B(L^2(\M,\psi)))$, as in \eqref{fix A}.

\subsection{Hoffman's bound}

One of the well-known spectral bounds in graph theory is the Hoffman's bound \cite{hoffman}.
This is a lower bound on the chromatic number of a graph using the largest and smallest eigenvalues of the adjacency matrix.
The classical bound is as follows:
If $G$ is an irreflexive classical graph whose adjacency matrix $A$ has eigenvalues $\lambda_{\max} = \lambda_1 \ge \lambda_2 \ge \ldots \ge \lambda_n = \lambda_{\min}$, then 
\begin{equation}
 1 +  \dfrac{\lambda_{\max}}{| \lambda_{\min} |}  \le  \chi(G).
 \end{equation}

\noindent We can prove a quantum version of this bound using the following result from linear algebra.

\begin{lem} \label{l2}
Let $A$ be a self-adjoint matrix, block partitioned as 
$\begin{bmatrix}
A_{11} & A_{12} & \ldots & A_{1n} \\
A_{21} & A_{22} & \ldots & A_{2n} \\
\vdots & \vdots & \vdots & \vdots \\
A_{n1} & A_{n2} & \ldots & A_{nn}
\end{bmatrix}$. 
Then,
\[ (n-1) \lambda_{\min}(A) + \lambda_{\max}(A) \le \sum_{i=1}^n \lambda_{\max}(A_{ii}),\]
where $\lambda_{\max}(\cdot)$ and $\lambda_{\min}(\cdot)$ represent the maximum and minimum eigenvalues of that matrix.

\begin{proof}
We start with the case $n=2$.
\noindent
Let $x = 
\begin{bmatrix}
x_1 \\
x_2
\end{bmatrix}$ be a normalized eigenvector ($\|x_1\|^2 + \|x_2\|^2 = 1$) corresponding to $\lambda_{\max}(A)$.
Define 
$y = 
\begin{bmatrix}
\frac{\|x_2\|}{\|x_1\|} x_1 \\
- \frac{\|x_1\|}{\|x_2\|} x_2
\end{bmatrix}$. Then, we have
\[ \lambda_{\max}(A) + \lambda_{\min}(A) \le 
\bra{x}A\ket{x} + \bra{y}A\ket{y} = 
\dfrac{\bra{x_1} A_{11} \ket{x_1}}{\|x_1\|^2} + \dfrac{\bra{x_2} A_{22} \ket{x_2}}{\|x_2\|^2} \le
\lambda_{\max}(A_{11}) + \lambda_{\max}(A_{22}).\]
The general case follows by induction on $n$.
\end{proof}

\end{lem}

\noindent The generalization of Hoffman's bound to quantum graphs is as follows:
\begin{thm}
Let $\G = (\M, \psi, A, S)$ be an irreflexive quantum graph and $\lambda_{\max} = \lambda_1 \ge \lambda_2 \ge \ldots \ge \lambda_{\dim(\M)} = \lambda_{\min}$ be all the eigenvalues of $A$. Then
\begin{equation} 
1 +  \dfrac{\lambda_{\max}}{| \lambda_{\min} |} \le \chi_q(\G).
\end{equation}

\begin{proof}
Let $\{P_k\}_{k=1}^c \subseteq \M \otimes \N$ be a $c$-quantum coloring of $\G$ and $\tilde{A} = A \otimes I_{\N}$. Partition $\tilde{A}$ as $[ \widehat{A}_{ab} ]_{a,b=1}^c$, as in proposition \ref{A tilda}.
Applying lemma \ref{l2}, we get 
\begin{equation} \label{eq5}
(c-1) \lambda_{\min}(\tilde{A}) + \lambda_{\max}(\tilde{A}) \le \sum_{i=1}^c \lambda_{\max}(\widehat{A}_{ii}).
\end{equation}
But $\widehat{A}_{ii} = 0$ for all $1 \le i \le c$. 
Hence equation \eqref{eq5} reduces to 
\[ (c-1) \lambda_{\min}(\tilde{A}) + \lambda_{\max}(\tilde{A}) \le 0.\]
Recall that $\lambda_{\min}(\tilde{A})  = \lambda_{\min}(A)$ and $\lambda_{\max}(\tilde{A}) = \lambda_{\max}(A)$.
So, we get $(c-1) \lambda_{\min}(A) + \lambda_{\max}(A) \le 0$.
On rearranging and taking minimum over all $c$, we get 
\[ 1 +  \dfrac{\lambda_{\max}(A)}{| \lambda_{\min} (A) |} \le \chi_q(\G).\]
\end{proof}

\end{thm}

\subsection{Lower bound using edge number}

In this section,  we prove a spectral lower bound on the quantum chromatic number using a quantum analogue for the number of edges in the graph.

For a classical graph $G$ with $n$ vertices and $m$ edges, it was shown \cite{edge} that
\begin{equation} \label{b3}
1 + \dfrac{2m}{2m - n \gamma_{\min}} \le \chi(G),
\end{equation}
where $\gamma_{\min}$ is the minimum eigenvalue of the signless Laplacian of $G$.
To prove a generalization of this bound to arbitrary quantum graphs $(\M, \psi, A, S)$, we first introduce a quantum analogue for $m, n$ and $\gamma_{\min}$.

Recall that the degree matrix for classical graphs is a diagonal matrix obtained from the action of the adjacency matrix on the all 1s vector.
This can be extended to quantum graphs as follows:

\begin{defn}
Let $\G = (\M,\psi,A, S)$ be a quantum graph and $\unit$ denote the unit in $\M$. Then the 
\textit{quantum degree matrix} of $\G$ is a linear operator $D \in B(L^2(\M))$ given by
\[ D: \M \longrightarrow \M \mbox{ as } x \mapsto x (A \unit), \forall x \in \M.\]
In other words, $D$ can be interpreted as $A \unit \in \M$ viewed as an element of $B(L^2(\M))$ under the \textit{right} regular representation.
\end{defn}

\begin{rem}The definition $D = A \unit$ was also used in \cite{randomqg} and \cite{matsuda}. The only difference in our case is that we view $D$ under the \textit{right} regular representation, instead of the usual left regular representation of $\M$. 
The advantage of using right regular representation is that $D$ then belongs to $\M'$.
\end{rem}

Our next goal is to define a quantum analogue for the ``number of edges" in the graph. To do that, we need the following result:

\begin{prop} \label{value of 2m}
Let $\M$ be a finite dimensional C*-algebra, equipped with its tracial $\delta$-form $\psi$. 
If $(\M,\psi,A, S)$ is a quantum graph with degree matrix $D$, then,
\begin{equation} \Tr(D) = \delta^2 \psi(A \unit) = \dim(S). \end{equation}

\begin{proof}
Let $P_S: B(L^2(\M)) \to B(L^2(\M))$ denote the orthogonal bimodule projection onto $S$. 
We can express $P_S$ as an element 
$\sum_{i=1}^t x_i \otimes y_i^{op} \in \M \otimes \M^{op}$, such that
$P_S (a \otimes b) = \sum_{i=1}^{t}  x_ia \otimes by_i$, for all $a,b \in \M$ using the correspondence mentioned in remark \ref{cbm}.
Now, $A = \delta^2 (\psi \otimes I)P_S(I \otimes \eta)$ implies
\begin{equation}
 A(\unit) = \delta^2 (\psi \otimes I) P_S (\unit \otimes \unit) = \delta^2 (\psi \otimes I) (\s{t} x_i \otimes y_i) 
 = \delta^2 \s{t} \psi(x_i) y_i.
 \end{equation}
Thus, \vspace{-0.5cm}
\begin{eqnarray*}
\psi(A \unit) & = & \psi( \delta^2 \s{t} \psi(x_i) y_i) =  \delta^2 \s{t} \psi(x_i) \psi(y_i) \\
& = & \delta^2 \s{t} \inner{x_i}{\unit}\inner{y_i}{\unit} \\
& = & \delta^2 \s{t} \inner{x_i \otimes y_i}{\unit \otimes \unit} 
=  \delta^2 \inner{ \s{t} x_i \otimes y_i}{\unit \otimes \unit} \\
& = & \delta^2 \inner{P_S}{I}, \mbox{       when viewed as operators on $B(L^2(\M))$} \\
& = & \delta^2 \dfrac{\Tr(P_S)}{\dim(B(L^2(\M)))}   \\
& = &  \dim(\M) \dfrac{\dim(S)}{\dim(\M)^2} = \dfrac{\dim(S)}{\dim(\M)}
\end{eqnarray*}
where we have used the fact that $\psi$ is a tracial state and $\delta^2 = \dim(\M)$.
Also, the trace on $B(L^2(\M))$ restricted to $\M$ (or $\M'$ by symmetry) is just $\dim(\M) \psi$.
So, 
\[ \Tr(D) = \dim(\M) \; \psi(A \unit). \]
Hence, $\Tr(D) = \dim(S)$.
\end{proof}

\end{prop}

We now define quantum analogues of some classical quantities:

\begin{defn} \label{q-defn}
Let $\G = (\M,\psi,A, S)$ be an irreflexive quantum graph with degree matrix $D$.
\begin{enumerate}
\item The \textit{quantum vertex number} for $\G$ is defined to be $\dim(\M)$. 
\item The \textit{quantum edge number} for $\G$ is defined to be
$\dfrac{\Tr(D)}{2} = \dfrac{\dim(S)}{2}.$
\item The \textit{Laplacian} of $\G$ is the linear operator $L = D - A \in B(L^2(\M))$.
\item The \textit{signless Laplacian} of $\G$ is the linear operator $Q = D + A \in B(L^2(\M))$.
\end{enumerate}
\end{defn}

For a classical irreflexive graph $G = (V,E)$, these definitions clearly coincide with the usual values. In particular, 
if $\G = (S_G, D_{|V|}, M_{|V|})$, then the quantum vertex number is $|V|$ and the quantum edge number is $|E|$ since $2 |E| = \sum_{v \in V} deg(v) = \Tr(D)$. 

\begin{rem} The quantum edge number need not be an integer in general. But for most purposes, we will only need $2m = \Tr(D) = \dim(S)$.
\end{rem}

We are now ready to prove a quantum version of the spectral bound in \eqref{b3}.

\begin{thm}
Let $\G = (\M,\psi, A, S)$ be an irreflexive quantum graph . Then 
\begin{equation}
1 + \dfrac{2m}{2m - n \gamma_{\min}} \le \chi_q(\G),
\end{equation}
 where $m$ is the quantum edge number, $n$ is the quantum vertex number and $\gamma_{\min}$ is the minimum eigenvalue of the signless Laplacian of $\G$, in the sense of definition \ref{q-defn}.
More precisely,
\begin{equation}
1 + \frac{\dim(S)}{\dim(S) -\dim(\M) \gamma_{\min}} \le \chi_q(\G).
\end{equation}

\begin{proof}
Let $\{P_k\}_{k=1}^c \subseteq \M \otimes \N$ be a $c$-quantum coloring of $\G$ and let 
$U$ be defined as in lemma \ref{pinching2}.
Then, \eqref{l7} can be rearranged as 
$ U^c (A \otimes I_{\N})(U^*)^c = - \sum_{k=1}^{c-1} U^k (A \otimes I_{\N}) (U^*)^k$.
Using $D - Q = - A$ and $U^c = I_{\M \otimes \N}$, we get 
\begin{eqnarray*}
 A \otimes I_{\N} & = & \sum_{k=1}^{c-1} U^k ((D-Q) \otimes I_{\N}) (U^*)^k \\
& = & \sum_{k=1}^{c-1} U^k (D \otimes I_{\N}) (U^*)^k - \sum_{k=1}^{c-1} U^k (Q \otimes I_{\N}) (U^*)^k \\ \label{eq1}
& = & (D \otimes I_{\N}) \sum_{k=1}^{c-1} U^k (U^*)^k - \sum_{k=1}^{c-1} U^k (Q \otimes I_{\N}) (U^*)^k \\ 
& = & (c-1) (D \otimes I_{\N}) - \sum_{k=1}^{c-1} U^k (Q \otimes I_{\N}) (U^*)^k
\end{eqnarray*}
where we have used the fact that $D \in \M'$ and hence $D \otimes I_{\N}$ commutes with $U \in \M \otimes \N$. 
Let $\N$ be represented in some $B(\h)$ and let $u$ denote a unit vector in $\h$ such that $\inner{u}{u} = 1$.
Further, let $\ket{\xi} = \unit \otimes u$ denote a column vector in $L^2(\M) \otimes \h$ and $\bra{\xi}$ denote its corresponding conjugate row vector. 
Multiplying the left and right most sides of the above equation by $\bra{\xi}$ from the left and by $\ket{\xi}$ from the right, we obtain
\begin{equation} \label{eq2}
\bra{\xi} A \otimes I_{\N} \ket{\xi} = (c-1)  \bra{\xi} D \otimes I_{\N} \ket{\xi}
- \sum_{k=1}^{c-1}  \bra{\xi} U^k (Q \otimes I_{\N}) (U^*)^k  \ket{\xi}.
\end{equation}
Now, 
$\bra{\xi}A \otimes I_{\N} \ket{ \xi} = 
\inner{\unit}{A \unit} \inner{u}{u}
= \psi( (A\unit)^*)= \psi( A\unit) 
= \dfrac{\dim(S)}{\dim(\M)}$, where we use the *-preserving property of $A$ (remark \ref{* preserve}) and proposition \ref{value of 2m}.
Similarly, $\langle \xi | D \otimes I_{\N} | \xi \rangle = \dfrac{\dim(S)}{\dim(\M)}$. 
To estimate the last term, recall that eigenvalues are invariant under unitary conjugation and tensoring with identity only changes their multiplicity. So, 
\begin{eqnarray*}
\gamma_{\min} 
& = &  \min \left \{ \bra{w} Q \ket{w} : w \in L^2(\M), \inner{w}{w} = 1 \right\} \\
& = & \min \left \{ \langle v | Q \otimes I_{\N} | v \rangle : v \in L^2(\M) \otimes \h, \inner{v}{v} = 1 \right \} \\
& = & \min \left \{ \langle v | U^k(Q \otimes I_{\N})(U^*)^k | v \rangle : v \in L^2(\M) \otimes \h, \inner{v}{v} = 1 \right \} \\
& \le & \langle \xi | U^k (Q \otimes I_{\N}) (U^*)^k | \xi \rangle, \;\; \forall k \in [c]. 
\end{eqnarray*}
Hence, \eqref{eq2} leads to 
\begin{equation}
\dfrac{\dim(S)}{\dim(\M)} \le (c-1) \dfrac{\dim(S)}{\dim(\M)} - (c-1) \gamma_{\min},
\end{equation} which upon rearranging yields
$1 + \dfrac{\dim(S)}{\dim(S) - \dim(\M) \gamma_{\min}} \le c$. Taking minimum over all $c$ , we get the desired bound. 

\end{proof}

\end{thm}

\subsection{Bound using the sum of square of eigenvalues}
In \cite{ando-lin}, it was proved that for a classical graph $G$,
\begin{equation}
1 + \max \left \{ \dfrac{s^+}{s^-} , \dfrac{s^-}{s^+} \right \} \le \chi(G),
\end{equation}
where $s^+$ is the sum of the squares of the positive eigenvalues of the adjacency matrix and 
$s^-$ is the sum of the squares of its negative eigenvalues.
In this section, we show that the above bound also works in the setting of quantum graphs. 
We first recall the following result from linear algebra, whose proof can be found in \cite{ando-lin}.

 \begin{lem} \label{l1}
 Let $X = [X_{ij}]_{i,j}^r$ and $Y = [Y_{ij}]_{i,j}^r$ be two positive semidefinite matrices conformally partitioned. 
 If $X_{ii} = Y_{ii}$ for $1 \le i \le r$ and $XY =0$, then $\Tr(X^*X) \le (r-1) \Tr(Y^*Y)$.
 \end{lem}

We now adapt the proof of the classical bound in \cite{ando-lin} to the quantum case.

\begin{thm} \label{bound3}
Let $\G = (\M, \psi, A, S)$ be an irreflexive quantum graph and 
$\lambda_1 \ge \lambda_2 \ge \ldots \ge \lambda_{\dim (\M)}$ be all the eigenvalues of $A$.
Let $s^+ = \sum_{\lambda_i > 0} (\lambda_i)^2$ and $s^- = \sum_{\lambda_i < 0} (\lambda_i)^2$. Then,
\begin{equation}
1 + \max \left \{ \dfrac{s^+}{s^-} , \dfrac{s^-}{s^+} \right \} \le \chi_q(\G).
\end{equation}

\begin{proof}
Let $\{P_k\}_{k=1}^c \subseteq \M \otimes \N$ be a $c$-quantum coloring of $\G$.
Further, let $\tilde{A} = A \otimes I_{\N}$ and let $\mu_1 \ge \mu_2 \ge \ldots \ge \mu_t$ be all the eigenvalues of $\tilde{A}$. 
Consider a spectral decomposition of $\tilde{A}$,
\begin{equation}
\tilde{A} = \s{t} \mu_i (v_i v_i^*), \mbox{ where } v_i \in L^2(\M) \otimes \h,
 \end{equation}
and write $\tilde{A} = \tilde{B} - \tilde{C}$, where 
\begin{equation} \label{eq3}
\tilde{B} = \sum_{ \mu_i > 0} \mu_i (v_i v_i^*) \hspace{1cm} \tilde{C} = \sum_{ \mu_i < 0} - \mu_i (v_i v_i^*).
\end{equation}
Suppose $N \subseteq B(\h)$ for some Hilbert space $\h$. Then,
\begin{equation} \label{eq4}
\Tr( {\tilde{B}}^* \tilde{B}) = \sum_{ \mu_i > 0} \mu_i^2 = \dim(\h) \; s^+ \mbox{ and }
\Tr( \tilde{C}^* \tilde{C}) = \sum_{ \mu_i < 0} \mu_i^2 = \dim(\h) \; s^- .
\end{equation}
Partition $\tilde{A}$ as $ [ \widehat{A}_{ab} ]_{a,b=1}^c$
as in proposition \ref{A tilda}.
Similarly, let 
\[ \tilde{B} = [\widehat{B}_{ab}]_{a,b=1}^c = \sum_{a,b =1}^c P_a \tilde{B} P_b
\mbox{ and }
\tilde{C} = [\widehat{C}_{ab}]_{a,b=1}^c = \sum_{a,b =1}^c P_a \tilde{C} P_b.\]
Now, $B$ and $C$ are positive semidefinite matrices that are conformally partitioned. 
Further, $\widehat{B}_{ii} = \widehat{C}_{ii}$  since $0 = P_i \tilde{A} P_i = P_i \tilde{B} P_i - P_i \tilde{C} P_i$ for all $1 \le i \le c$.
Also $\tilde{B}\tilde{C} = \tilde{C}\tilde{B} = 0$.
So, by lemma \ref{l1} and \eqref{eq4}, it follows that 
$\dfrac{s^+}{s^-} \le c - 1$ and $\dfrac{s^-}{s^+} \le c - 1$. Taking minimum over all $c$, we get 
$1 + \max \left \{ \dfrac{s^+}{s^-} , \dfrac{s^-}{s^+} \right \} \le \chi_q(\G)$.
\end{proof}

\end{thm}

\subsection{Inertial lower bound}

In this section, our goal is to generalize the following inertial bound \cite{inertial} to quantum graphs:
\begin{equation}
1 + \max \left \{  \dfrac{n^+}{n^-},  \dfrac{n^-}{n^+}    \right \} \le \chi(G), 
\end{equation}
where $(n^+, n^0, n^-)$ is the inertia of $G$. We begin with defining the inertia of a quantum graph:

\begin{defn}
Let $\G = (\M, \psi, A, S)$ be a quantum graph and 
 $\lambda_1 \ge \lambda_2 \ge \ldots \ge \lambda_{\dim(\M)}$ denote the eigenvalues of $A$.
The \textit{inertia} of $\G$ is the ordered triple $(n^+, n^0, n^-)$, where $n^+$, $n^0$ and $n^-$ are the numbers of positive, zero and negative eigenvalues of $A$ including multiplicities.
\end{defn}

\begin{thm}
Let $\G = (\M, \psi, A, S)$ be an irreflexive quantum graph with inertia $(n^+, n^0, n^-)$.
Then, 
\begin{equation}
1 + \max \left \{  \dfrac{n^+}{n^-},  \dfrac{n^-}{n^+}    \right \} \le \chi_q(\G).
\end{equation}
\end{thm}

\begin{proof}
Let $\{P_k\}_{k=1}^c \subseteq \M \otimes \N$ be a $c$-quantum coloring of $\G$.
Let $U$ be defined as in lemma \ref{pinching2} and $\tilde{A}, \tilde{B}$ and $\tilde{C}$ be defined as in the proof of theorem \ref{bound3}. Then, we have
\begin{equation} \label{l5}
\sum_{k=1}^{c-1} U^k \tilde{B} (U^*)^k - \sum_{k=1}^{c-1} U^k \tilde{C} (U^*)^k  = \sum_{k=1}^{c-1} U^k \tilde{A} (U^*)^k = -\tilde{A} = \tilde{C} - \tilde{B},
\end{equation}

Note that $\tilde{B}$ and $\tilde{C}$ are positive definite operators with $rank(\tilde{B}) = n^+$ and $rank(\tilde{C}) = n^-$. 
Further let 
\[ P^+ = \sum_{\mu_i > 0} v_i v_i^* \mbox{ and }  P^- = \sum_{\mu_i < 0} v_i v_i^* \]
denote the orthogonal projectors onto the subspaces spanned by the eigenvectors corresponding to the positive and negative eigenvalues of $\tilde{A}$ respectively. Observe that
$\tilde{B} = P^+ \tilde{A} P^+$ and $\tilde{C} = - P^- \tilde{A} P^-$.
Multiplying \eqref{l5} by $P^-$ on both sides, we obtain:
\begin{equation} \label{l6}
P^- \sum_{k=1}^{c-1} U^k \tilde{B} (U^*)^k P^- - P^- \sum_{k=1}^{c-1} U^k \tilde{C} (U^*)^k P^- = C
\end{equation}

Now we use the fact that if $X,Y$ are two positive definite matrices such that $X-Y$ is positive definite, then $rank(X) \ge rank(Y)$. By applying this to \eqref{l6}, we get 
\[ rank(P^- \sum_{k=1}^{c-1} U^k \tilde{B} (U^*)^k P^-) \ge rank(C). \]
Recall that the rank of a sum is less than or equal to the sum of the ranks of the summands, and that the rank of a product is less than or equal to the minimum of the ranks of the factors. So, we get $(c-1)n^+ \ge n^-$. 
Similarly, it can be shown that 
$(c-1)n^- \ge n^+$. 
Hence, $ \max \left \{  \dfrac{n^+}{n^-},  \dfrac{n^-}{n^+}    \right \} \le c-1$.
Taking minimum over all $c$, we get the desired bound.
\end{proof}

\subsection{Bound using maximum eigenvalue of the Laplacian and signless Laplacian}

Let $L$ and $Q$ denote the Laplacian and signless Laplacian of $\G = (\M, \psi, A, S)$ in the sense of definition \ref{q-defn}. Further, let $\lambda_{\max}, \theta_{\max}$ and $\gamma_{\max}$ denote the largest eigenvalue of $A, L$ and $Q$ respectively. Then 
\begin{equation}
1 + \dfrac{\lambda_{\max}}{\lambda_{\max} - \gamma_{\max}+ \theta_{max}}  \le \chi_q(\G).
\end{equation}

Like the previous cases, this bound can also be shown by adapting the classical proof \cite{kolotilina} and applying lemma \ref{pinching2}.

\section{Illustration} \label{application}
In this section, we illustrate the tightness of these bounds in the case of complete quantum  graphs.
Let $K_{\M}$ denote the irreflexive complete quantum  graph on $(\M, \psi)$. The quantum adjacency matrix in this case is given by
$A = \delta^2 \psi( \cdot) \unit - I$. 
For $x \in \M$, we have
\begin{eqnarray*}
A(x) & = & \delta^2 \psi(x) \unit - I \\
& = & (\dim \M) \inner{x}{\unit} \unit - I \\
& = & (\dim \M) P_{\unit}(x) - I,
\end{eqnarray*}
where $P_{\unit}: \M \to \M$ denotes the orthogonal projection  onto 1, given by $x \mapsto \inner{x}{\unit} \unit$. 
Since $P_{\unit}$ is a rank-1 projection, its spectrum is precisely $\{0,1\}$, where $0$ has a multiplicity of $\dim(\M) - 1$. Using functional calculus, we get 
\begin{equation}
\sigma(A) = \{ \dim(\M) - 1, \;\; -1 \},
\end{equation} 
where $-1$ has a multiplicity of $\dim(\M) - 1$.
Similarly, we get 
\begin{equation}
\sigma(Q) = \{ 2 \dim(\M) - 2, \; \dim(\M) - 2 \}, 
\end{equation} 
where $\dim(\M) - 2$ has a multiplicity of $\dim(\M) - 1$, and
\begin{equation}
\sigma(L) = \{ \dim(\M) , 0 \}, 
\end{equation} 
where $\dim(\M)$ has a multiplicity of $\dim(\M) - 1$.

\noindent Thus, for an irreflexive complete quantum graph, we have:
\vspace{-0.4cm}
\begin{itemize}[noitemsep]
\item $\lambda_{\max} = \dim \M - 1$, $\lambda_{\min} = - 1$
\item $\gamma_{\max} = 2 \dim(\M) - 2$, $\gamma_{\min} = \dim \M - 2$
\item $\theta_{\max} = \dim \M$
\item $s^+ = (\dim(\M) - 1)^2$, $s^- = \dim(\M) - 1$
\item $n^+ = 1$, $n^- = \dim(\M) -1$
\item $2m = \dim(\M)^2 - \dim(\M)$
\end{itemize}

\noindent On applying these to theorem \ref{main result}, we see that all the five spectral bounds give the same result, namely:
\begin{equation}
\dim(\M) \le \chi_q(K_{\M}).
\end{equation}
The reverse inequality $\chi_q(K_{\M}) \le \dim(\M)$ was proved in \cite{bgh}, and $\chi_q(K_{\M}) = \dim(\M)$.
So, we conclude that all the bounds in theorem \ref{main result} are tight in the case of complete quantum graphs.

\section{Conclusion and future directions} \vspace{-0.3cm}
In this work, we have shown that several spectral lower bounds for the chromatic number of classical graphs 
are also lower bounds for the classical and quantum chromatic numbers of quantum graphs. 
We believe that quantum graph spectral theory is a promising field of study.
As a next step, it would be interesting to find bounds that exhibit a separation between the different variants of chromatic numbers of quantum graphs. Alternatively, investigating examples of quantum graphs that show a separation between these spectral bounds would also be helpful. We hope to explore these in a future work.

\section{Acknowledgments} \vspace{-0.3cm}
The author is grateful to her PhD supervisor, Michael Brannan, for his valuable guidance and insights on this project. The author would also like to thank Larissa Kroell for useful discussions on quantum graph theory. This work was partially supported by NSF grants DMS-2000331 and DMS-1700267.

\section{Appendix} \label{appendix}
\vspace{-0.3cm}

Let $\M$ be a finite dimensional C*-algebra, equipped  with its tracial $\delta$-form $\psi$.
The properties of a quantum graph on $\M \subseteq B(L^2(\M, \psi))$ in the different perspectives is summarized in the following table. 

Here, $p = \sum_{i=1}^t a_i \otimes b_i \in \M \otimes \M^{op}$ and $m, \sigma$ denote the multiplication map and swap map on $\M \otimes \M^{op}$ respectively. Further, $\h = L^2(\M, \psi)$, $T \in B(L^2(\M))$, $ \xi \in L^2(\M)$ and $x,y \in \M'$.

\def\arraystretch{2}
\begin{center}

\begin{tabular}{| p{1.8cm} || p{1.8cm} | p{1.8cm} | p{2.5cm} | p{3.4cm} | p{3.2cm} |}
\hline

\textsc{property} & \textsc{classical graph} & $S \subseteq B(\h)$ & $p \in \M \otimes \M^{op}$ & $A: \M \to \M$ & $P: B(\h) \to B(\h)$ \\ \hline \hline

Bimodule structure & Relations on a set & $\M'S\M' \subseteq S$ & 
$\sum_{i} a_i (xTy) b_i = x(\sum_{i} a_i T b_i )y$ & 
$m(A \otimes xTy) m^* = \newline x (m(A \otimes T ) m^*) y$ & 
$P(xTy) = xP(T)y$ \\ \hline

Schur \newline idempotent & $A \in M_n(\{0,1\})$ & $A \in S$ & $p^2 = p$ & $m(A \otimes A)m^* = \delta^{2} A$  
& $P^2 = P$  \\ \hline

Reflexive & $I \in S_{G}$ & $\M' \subseteq S$ & $m(p) = \unit$ & $m(A \otimes I)m^* =  \delta^{2} I$ & $P(I) = I$ \\ \hline 

Irreflexive & $ I \notin S_G$ & $\M' \perp S$ & $m(p) = 0$ & $m(A \otimes I)m^* = 0$ & $P(I) = 0$ \\ \hline 

Undirected & $A = A^T$ & $S = S^*$ & $\sigma(p) = p$ & 
$(I \otimes \eta^*m )(I \otimes A \otimes I)\newline ( m^* \eta \otimes I) = A $ 
\newline Alternatively,  $A(\xi^*) = [A^*(\xi)]^* $ 
& $P^*(T) = P(T^*)^*$,  \newline {\small ( * denotes adjoint as an operator on Hilbert spaces)}  \\ \hline

Self \newline adjoint & $A = A^*$ &   & $\sigma(p) = p^*$  & $A(\xi) = A^*(\xi)$ &  
$P(T^*) = P(T)^*$ \\ \hline

Real & $A = \overline{A}$ &  & $p^* = p$ & $A(\xi^*) = (A(\xi))^*$ & $P^*(T) = P(T)$ \\ \hline

Positivity & $A$ is C.P &  & $p$ is positive (i.e. $p = g^*g$) & $A$ is completely \newline positive (C.P) &
 $P$ is positive \newline (i.e. $P = G^*G$) \\ \hline

\end{tabular}

\end{center}

In particular, for undirected quantum graphs:
\begin{eqnarray*}
P^2 = P = P^* & \iff & p^2 = p = p^* \\
& \iff & A \mbox{ is Schur-idempotent and real } \\
& \iff & A \mbox{ is Schur-idempotent and self-adjoint.}
\end{eqnarray*}

\bibliographystyle{amsalpha}
\bibliography{spectral_bounds}

\end{document}